\documentclass{amsart}
\usepackage{amsthm,amsmath,amsmath,amssymb,amsbsy,graphicx,color,epstopdf}

\theoremstyle{plain}
\newtheorem{theorem}{Theorem}
\newtheorem{corollary}[theorem]{Corollary}
\newtheorem{proposition}[theorem]{Proposition}
\newtheorem{lemma}[theorem]{Lemma}

\theoremstyle{definition}

\newtheorem{example}[theorem]{Example}

\theoremstyle{remark}
\newtheorem{remark}[theorem]{Remark}

\DeclareMathOperator{\diag}{diag}
\DeclareMathOperator{\ord}{ord}
\DeclareMathOperator{\im}{im}

\newcommand{\SSS}{\mathbb{S}}

\numberwithin{equation}{section}
\numberwithin{theorem}{section}

\makeatletter
\def\@strippedMR{}
\def\@scanforMR#1#2#3\endscan{%
  \ifx#1M\ifx#2R\def\@strippedMR{#3}%
  \else\def\@strippedMR{#1#2#3}%
  \fi\fi}
\renewcommand\MR[1]{\relax
  \ifhmode\unskip\spacefactor3000 \space\fi
  \@scanforMR#1\endscan
  MR\MRhref{\@strippedMR}{\@strippedMR}}
\makeatother 

\begin{document}

\title{On a parametrization of positive semidefinite matrices with zeros} 
\date{\today} 
\author{Mathias Drton}
\address{Department of Statistics, The University of Chicago, Chicago,
  Illinois, U.S.A.}  
\email{drton@uchicago.edu}

\author{Josephine Yu}
\address{School of Mathematics, Georgia Institute of Technology, Atlanta, Georgia,
  U.S.A.}   
\email{josephine.yu@math.gatech.edu}

\begin{abstract} 
  We study a class of parametrizations of convex cones of positive
  semidefinite matrices with prescribed zeros. Each such cone
  corresponds to a graph whose non-edges determine the prescribed
  zeros. Each parametrization in this class is a polynomial map
  associated with a simplicial complex supported on cliques of the
  graph. The images of the maps are convex cones, and the maps can
  only be surjective onto the cone of zero-constrained positive
  semidefinite matrices when the associated graph is chordal and the
  simplicial complex is the clique complex of the graph. Our main
  result gives a semi-algebraic description of the image of the
  parametrizations for chordless cycles. The work is motivated by the
  fact that the considered maps correspond to Gaussian statistical
  models with hidden variables.
\end{abstract}


\maketitle

\section{Introduction}
\label{sec:introduction}

For a positive integer $m$, let $[m]=\{1,\dots,m\}$.  Denote the power
set of $F\subseteq[m]$ by $2^{F}$.  A collection of subsets
$\Delta\subseteq 2^{[m]}$ is a simplicial complex if
$2^F\subseteq\Delta$ for all $F\in \Delta$.  The elements of $\Delta$
are called faces and the inclusion-maximal faces are the facets.  The
ground set of $\Delta$ is the union of its faces.  The underlying
graph $G(\Delta)$ is the simple undirected graph with the ground set
as vertex set and the 2-element faces as edges.  All simplicial
complexes appearing in this paper are assumed to have ground set $[m]$
  and, thus, all underlying graphs have vertex set $[m]$.  We make this
assumption explicit by speaking of a simplicial complex on $[m]$.

Let $\SSS^m$ be the $m(m+1)/2$ dimensional vector space of symmetric
$m\times m$ matrices.  For an undirected graph $G$ with vertex set
$V(G)=[m]$ and edge set $E(G)$, define the $|E(G)| + m$ dimensional subspace
\[
\SSS^m(G) \quad=\quad\{\, \Sigma=(\sigma_{ij})\in \SSS^m \::\:
\sigma_{ij} = 0 \; \text{if} \; i\not= j \; \text{and}\;
\{i,j\}\not\in E(G)\,\}
\]
containing the symmetric matrices with zeros at the non-edges of $G$.
Let $\SSS_{\succeq 0}^m\subset \SSS^m$ be the convex cone of positive
semidefinite matrices and $\SSS_{\succeq 0}^m(G)=\SSS_{\succeq
  0}^m\cap \SSS^m(G)$ the convex subcone of matrices with zeros
prescribed by the graph.

This paper is concerned with particular parametrizations of the
graphical cone $\SSS_{\succeq 0}^m(G)$.  For a subset $\Delta\subseteq
2^{[m]}$, define the polynomial map
\[
\phi_\Delta : \prod_{F \in \Delta} \mathbb{R}^{|F|} \rightarrow
\mathbb{S}_{\succeq 0}^m
\]
given by
\[
\phi_\Delta(\gamma) = \Gamma(\gamma)\Gamma(\gamma)^T,
\]
where, for $\gamma = (\gamma_{i, F} : F \in \Delta, i \in F)$,
the $[m] \times \Delta$ matrix $\Gamma(\gamma)$ has entries
\begin{equation}
\label{eq:def-Gamma-gamma}
\Gamma(\gamma)_{i,F} = 
\begin{cases}
\gamma_{i,F} & \mbox{ if } i \in F,\\
0 & \mbox{ otherwise}.
\end{cases}
\end{equation}
The coordinates of the map are
\begin{equation}
  \label{eq:phi-coords}
  \phi_\Delta(\gamma)_{ij} = 
  \sum_{F\in\Delta\,:\, i,j\in F} \gamma_{i,F}\gamma_{j,F}, \qquad
  i,j\in [m].
\end{equation}
In particular, the diagonal coordinates
\begin{equation}
  \label{eq:phi-coords-diag}
  \phi_\Delta(\gamma)_{ii} = 
  \sum_{F\in\Delta\,:\, i\in F} \gamma_{i,F}^2, \qquad i\in[m],
\end{equation}
are sums of squares, which implies that $\phi_\Delta$ is a proper map,
that is, compact sets have compact preimages under $\phi_\Delta$.

We will be interested in the situation when $\Delta$ is a simplicial
complex on $[m]$.  In this case, the map $\phi_\Delta$ is never
injective.  It has fibers (preimages) of positive dimension unless the
underlying graph is the empty graph.

\begin{lemma}
  \label{lem:full-dim}
  For any simplicial complex $\Delta$ on $[m]$ with underlying graph
  $G=G(\Delta)$, the image of $\phi_\Delta$ is a closed
  full-dimensional semi-algebraic subset of $\mathbb{S}_{\succeq
    0}^m(G)$.
\end{lemma}
\begin{proof}
  If $i\not= j$ and $\{i,j\}$ is not an edge of $G$, then no face of
  $\Delta$ contains both $i$ and $j$.  Hence, by
  (\ref{eq:phi-coords}), the image is a subset of $\mathbb{S}_{\succeq
    0}^m(G)$.  The image is semi-algebraic because $\phi_\Delta$ is a
  polynomial map, and it is closed because $\phi_\Delta$ is proper.
  
  If $\Delta'\subset\Delta$ is another simplicial complex with the
  same underlying graph then the image of $\phi_{\Delta'}$ is
  contained in the image of $\phi_{\Delta}$.  To show full dimension,
  we may thus assume that $\Delta$ is the complex whose facets are the
  edges of $G$.  Using the shorthand $\gamma_i=\gamma_{i,\{i\}}$ and
  $\gamma_{ij}=\gamma_{i,\{i,j\}}$ in this special case, the non-zero
  coordinates of $\phi_\Delta$ are
  \[
  \phi_\Delta(\gamma)_{ij}=
  \begin{cases}
    \gamma_{i}^2 + \sum_{k\in [m]: \{i,k\}\in\Delta}
      \gamma_{ik}^2 & \mbox{ if } i = j,\\
      \gamma_{ij}\gamma_{ji} & \mbox{ if } i\not= j.   
  \end{cases}
  \]
  It is evident that there are no algebraic relations among these
  coordinates and, thus, the image is full-dimensional.
\end{proof}

\begin{example}
  \label{ex:three-chain}
  Let $\Delta$ be the simplicial complex whose facets are the edges
  $\{1,2\}$ and $\{2,3\}$ of a three-chain.   We
  have that
  \begin{equation*}
    \Gamma(\gamma) =    \begin{pmatrix}
      \gamma_1 &0&0&\gamma_{12}&0\\
      0&\gamma_2&0&\gamma_{21}&\gamma_{23}\\
      0&0&\gamma_3&0&\gamma_{32}
    \end{pmatrix}
  \end{equation*}
  and
  \begin{equation*}
    \label{eq:cov-mx-2}
    \phi_\Delta(\gamma) = 
    \begin{pmatrix}
      \gamma_1^2+\gamma_{12}^2 & \gamma_{12}\gamma_{21} & 0\\
      \gamma_{12}\gamma_{21} & \gamma_{2}^2+\gamma_{21}^2+\gamma_{23}^2
      &\gamma_{23}\gamma_{32} \\ 
      0 & \gamma_{23}\gamma_{32} & \gamma_3^2+\gamma_{32}^2
    \end{pmatrix}.
  \end{equation*}
  It can be shown that $\phi_\Delta$ is a surjective map onto the
  entire cone $\SSS_{\succeq 0}^3(G)$, which here comprises the
  tridiagonal positive semidefinite matrices.  The surjectivity claim
  holds as a special case of Corollary~\ref{cor:surjective}.  \qed
\end{example} 

As we describe in more detail in Section~\ref{sec:statistical-models},
the motivation for considering the parametrization $\phi_\Delta$ comes
from statistics.  The graphical cones $\SSS_{\succeq 0}^m(G)$
correspond to statistical models for the multivariate normal
distribution; see \cite[\S2]{DrtonSJS2007} and references therein.
The parametrization $\phi_\Delta$ is particularly useful for tackling
statistical problems in covariance graph models, which treat the cone
$\SSS_{\succeq 0}(G)$ as a set of covariance matrices.  The
parametrization can be regarded as arising from constructions
involving hidden or latent variables
\cite{CoxWermuth1996,Richardson2002}.  This connection can be
exploited in particular for computation of maximum likelihood
estimates and construction of prior distributions for Bayesian
inference \cite{Barber:2008,Palomo2007}.  It also allows one to
simplify the study of algebraic properties of graphical models based
on mixed graphs; see \cite{Sullivant:trek}.

In Example~\ref{ex:three-chain}, the map $\phi_\Delta$ is surjective.
However, it is known that surjectivity need not always hold.  The
following example has been given in the literature.

\begin{example}
  \label{ex:three-cycle}
  Let $\Delta$ be the simplicial complex with facets $\{1,2\}$,
  $\{1,3\}$ and $\{2,3\}$, and the complete graph $K_3$ as
  underlying graph.  Now,
  \begin{equation*}
    \label{eq:cov-mx-3}
    \phi_{\Delta}(\gamma)=
    \begin{pmatrix}
      \gamma_1^2+\gamma_{12}^2+\gamma_{13}^2 & \gamma_{12}\gamma_{21}
      & \gamma_{13}\gamma_{31}\\ 
      \gamma_{12}\gamma_{21} &
      \gamma_{2}^2+\gamma_{21}^2+\gamma_{23}^2 &\gamma_{23}\gamma_{32}
      \\ 
      \gamma_{13}\gamma_{31} & \gamma_{23}\gamma_{32} &
      \gamma_3^2+\gamma_{31}^2+\gamma_{32}^2
    \end{pmatrix}.
  \end{equation*}
  Suppose we are given a positive definite matrix
  $\Sigma=(\sigma_{ij})$ in $\SSS^3_{\succeq 0}(K_3)=\SSS^3_{\succeq 0}$.
  Define the correlation matrix $R=(\rho_{ij})$ with entries
  $\rho_{ij}=\sigma_{ij}/\sqrt{\sigma_{ii}\sigma_{jj}}$.
  The matrix $R$ is obtained by multiplying $\Sigma$ from the left and
  right with the diagonal matrix that has the entries
  $1/\sqrt{\sigma_{ii}}$ on the diagonal.  It follows that $\Sigma$ is
  in the image of $\phi_\Delta$ if and only if $R$ is in the image.
  For $R$ to be in the image, however, it needs to hold that
  \begin{equation}
    \label{eq:tsr-rho-min}
    \min\left\{\rho_{12},\rho_{13},\rho_{23}\right\} \le \frac{1}{\sqrt{2}};
  \end{equation}  
  see \cite{SpirtesEtAl1998}.  Clearly, there are positive definite
  matrices in $\SSS^3_{\succeq 0}$ whose correlation matrices do not obey
  this condition. 
  
  Our Theorem~\ref{thm:semi-alg-cycle} applies to this example and
  gives a semi-algebraic description of the image of $\phi_\Delta$.
  This description reveals that a positive definite matrix is in the
  image if and only if its correlation matrix $R$ satisfies
  \begin{equation}
    \label{eq:det-R-12}
    1-\rho_{12}^2-\rho_{13}^2-\rho_{23}^2
    -2\rho_{12}\rho_{13}\rho_{23} \ge 0. 
  \end{equation}
  If $\rho_{12},\rho_{13},\rho_{23}> 1/2$, then the left hand side in
  (\ref{eq:det-R-12}) is smaller than $1-3/4-2/8 = 0$.  Hence, one may
  replace $1/\sqrt{2}$ by $1/2$ in the necessary condition in
  (\ref{eq:tsr-rho-min}), which can also be seen directly.  If $R =
  \phi_{\Delta}(\gamma)$ has diagonal entries $1$, then summing the
  diagonal entries gives
  $$
  3 = \sum_{i=1}^3 \gamma_i^2
  + \sum_{1\leq i < j \leq 3}(\gamma_{ij}^2 + \gamma_{ji}^2) =
  \sum_{i=1}^3 \gamma_i^2 + \sum_{1\leq i < j \leq 3}(\gamma_{ij} -
  \gamma_{ji})^2 + \sum_{1\leq i < j \leq 3}2 \gamma_{ij}
  \gamma_{ji},
  $$
  so we must have $\rho_{ij} = \gamma_{ij}\gamma_{ji} \leq 1/2$ for
  some $i,j$.
  \qed
\end{example}

This paper explores in detail the images of the maps $\phi_\Delta$, which
we denote by $\im(\phi_\Delta)$.  In Section~\ref{sec:convexity}, we show
that the image is always a convex cone and we describe its extreme rays.
In Section~\ref{sec:surjectivity}, we prove that surjectivity of the map
can only be achieved if $\Delta$ is the clique complex of a chordal (or
decomposable) graph.  Section~\ref{sec:sumbatrices-schur-complements}
collects results relevant for passing to submatrices and Schur complements.
In Section~\ref{sec:chordless-cycles}, we derive the semi-algebraic
description of the image when the underlying graph is a chordless cycle.
The connection to statistical models is reviewed in
Section~\ref{sec:statistical-models}.

\section{Convexity}
\label{sec:convexity}

The set $\mathbb{S}_{\succeq 0}^m$ of positive semidefinite $m\times m$
matrices forms a full-dimensional convex cone in the $m(m+1)/2$
dimensional vector space of $m \times m$ symmetric matrices.  A ray of
$\mathbb{S}_{\succeq 0}^m$ is the set of non-negative scalar multiples of
some non-zero matrix in $\mathbb{S}_{\succeq 0}^m$.  An extreme ray is a
ray that cannot be written as a positive linear combination of two
distinct rays.  The extreme rays of $\mathbb{S}_{\succeq 0}^m$ are given
by the positive semidefinite matrices of rank 1.  Hence,
$\mathbb{S}_{\succeq 0}^m$ is the convex hull of its rank 1 elements.

For $F\subseteq [m]$, let $\SSS_{\succeq 0}^m(F)$ be the convex cone
of positive semidefinite matrices that have zeros outside the $F\times
F$ submatrix.

\begin{theorem}
  \label{thm:convexity}
  For any simplicial complex $\Delta$ on $[m]$, the image of
  $\phi_\Delta$ is a convex cone.  The matrices on the extreme rays of
  the image are the rank one matrices that are in $\SSS_{\succeq
    0}^m(F)$ for some face $F \in \Delta$. In other words, $\im(\phi_{\Delta}) = \sum_{F \in \Delta} \SSS_{\succeq 0}^m(F)$.

\end{theorem}

\begin{proof}
  Elements of the image of the map $\phi_\Delta$ are of the form
  \[
  \sum_{F \in \Delta} \Gamma(\gamma)_F \Gamma(\gamma)_F^T,
  \]
  where $\Gamma(\gamma)_F$ is the column of $\Gamma(\gamma)$
  corresponding to face $F$.  This column can be any vector in
  $\mathbb{R}^m$ that has $i$-th entry zero for each $i \notin F$.  It
  is clear that the image of $\phi_\Delta$ is closed under positive
  scaling.  We will show that it is closed under addition, by
  induction on the maximal cardinality of a face in $\Delta$.
  
  If all faces have size 1 then the image of $\phi_\Delta$ consists of
  all positive semidefinite diagonal matrices and is convex.  Let $F$
  be a facet of $\Delta$ and suppose it has cardinality at least 2.
  Consider the matrix
  \[
  \Sigma = \Gamma(\gamma)_F  \Gamma(\gamma)_F^T +  \Gamma(\gamma')_F
  \Gamma(\gamma')_F^T 
  \]
  and its Cholesky decomposition $\Sigma = L L^T$, where $L$ is a
  lower triangular matrix.  Since $\Sigma\in\SSS_{\succeq 0}^m(F)$,
  each column of the Cholesky factor $L$ has support in $F$.
  In fact only the first column of $L$ may have support equal to $F$; denote
  this column by $L_1$.  All other columns of $L$ have support
  strictly smaller than $F$.  These smaller supports correspond to
  subfaces of $F$, so they are in $\Delta$.  Hence, $\Sigma$ is the
  sum of $L_1 L_1^T$ and an element in the image of $\phi_{\Delta
    \backslash \{F\}}$.  (Removing a facet leaves us with
  another simplicial complex.)  Repeating this process for all other
  faces of maximal cardinality in $\Delta$ and using the inductive
  hypothesis, we see that the image of $\phi_\Delta$ is closed under
  addition.
  
  Suppose a non-zero matrix $\Sigma$ is on an extreme ray of the
  convex cone $\im(\phi_\Delta)$.  Then $\Sigma = \Sigma_1 + \Sigma_2$
  for some non-zero and distinct matrices $\Sigma_1,\Sigma_2$ in the
  same cone implies that both $\Sigma_1$ and $\Sigma_2$ are scalar
  multiples of $\Sigma$.  From the definition, any element in the
  image of $\phi_\Delta$ is a sum of rank one matrices in it, so only
  rank one matrices can be on the extreme rays.  Moreover, any rank
  one positive semidefinite matrix is on an extreme ray of
  $\SSS_{\succeq 0}^m$, so it is also on an extreme ray of the convex
  subcone $\im(\phi_\Delta)$ that contains it.  A rank one matrix in
  $\im(\phi_\Delta)$ is of the form $v v^T$ for some vector $v \in
  \mathbb{R}^m$ whose support $F$ is a face of $\Delta$.  Hence,
  $vv^T\in\SSS_{\succeq 0}^m(F)$.
\end{proof}

A clique in an undirected graph $G$ with vertex set $[m]$ is a subset
$F\subseteq [m]$ such that for any pair of distinct vertices $i, j\in
F$,  $\{i,j\}$ is in $E(G)$.  The set of all cliques in $G$ forms a simplicial complex on $[m]$ and is called the {\em clique complex} of $G$.

\begin{corollary}
  Let $\Delta$ be a simplicial complex on $[m]$ with underlying graph
  $G$.  Then the extreme rays of the image of $\phi_\Delta$ consist of
  all rank one matrices in $\mathbb{S}_{\succeq 0}^m(G)$ if and only if
  $\Delta$ consists of all the cliques in $G$.
\end{corollary}

\section{Surjectivity}
\label{sec:surjectivity}

The maximal rank of a matrix lying on an extreme ray of
$\mathbb{S}_{\succeq 0}^m(G)$ is called the {\em sparsity order} of
the graph $G$ and denoted $\ord(G)$.  
A subgraph $H$ of $G$ is called an {\em induced subgraph} if for all pairs of vertices $i,j$ in $H$, $\{i,j\}\in E(H) \iff \{i,j\} \in E(G)$.  A graph is called {\em chordal} if it does not contain any chordless cycle of size more than three as an induced subgraph.
The following results are known
in the literature \cite{AHMR, HPR, Laurent2001}.

\begin{theorem}
  For a graph $G$ with $m$ vertices,
  \begin{itemize}
  \item[(i)] $1 \leq \ord(G) \leq m-2$,
  \item[(ii)] $\ord(G) = 1$ if and only if $G$ is chordal,
  \item[(iii)] $\ord(G) = m-2$ if and only if $m \leq 3$ or  $G$ is a chordless cycle, and
  \item[(iv)] if $H$ is an induced subgraph of $G$, then $\ord(H) \leq \ord(G)$.
  \end{itemize}
\end{theorem}

These results readily allow one to characterize when the
parametrization $\phi_\Delta$ fills all of the graphical cone
$\SSS_{\succeq 0}^m(G)$.

\begin{corollary}
  \label{cor:surjective}
  Let $\Delta$ be a simplicial complex and $G$ a graph on $[m]$.  The map
  $\phi_\Delta$ is surjective onto $\SSS_{\succeq 0}^m(G)$ if
  and only if $G$ is chordal and $\Delta$ is its clique complex.
\end{corollary}

\begin{proof}
  {\em (Sufficiency)} If $\Delta$ contains all cliques in $G$, then
  $\im(\phi_\Delta)$ contains all rank one matrices in $\SSS_{\succeq
    0}^m(G)$.  If $G$ is chordal, then its sparsity order is one, so
  $\SSS_{\succeq 0}^m(G)$ is generated by rank one matrices in it.
  Hence, $\im(\phi_\Delta) = \SSS_{\succeq 0}^m(G)$ and $\phi_\Delta$
  is surjective.
  
  {\em (Necessity)} First note that the image of $\phi_\Delta$ is a
  subset of $\SSS_{\succeq 0}^m(G)$ only if all sets in $\Delta$ are
  cliques of $G$.
  
  Let $\Delta$ be the clique complex of $G$.  If $G$ is not
  chordal, then there is an induced subgraph that is a chordless cycle
  of size at least 4.  So $\ord(G) \geq 2$, and there is an extreme
  ray of $\SSS_{\succeq 0}^m(G)$ containing matrices of rank at least
  two.  This ray is not in the convex cone $\im(\phi_\Delta)$, so
  $\phi_\Delta$ is not surjective.  It follows that $\phi_{\Delta'}$
  is not surjective for any (arbitrary) subset $\Delta'$ of $\Delta$.
  
  Suppose $\Delta$ does not contain a clique $F$ in
  $G$.  Let $v \in \mathbb{R}^m$ be a vector with support $F$.  Then
  $v v^T$ is a rank one element of $\SSS_{\succeq 0}^m(G)$.  It lies
  in an extreme ray of $\SSS_{\succeq 0}^m(G)$ because it lies in an
  extreme ray of the larger cone $\SSS_{\succeq 0}^m$.  Hence, it
  cannot be written as a sum of other elements in $\SSS_{\succeq
    0}^m(G)$, so it is not in $\im(\phi_\Delta)$, and $\phi_\Delta$ is
  not surjective.
\end{proof}

\begin{remark}
  The sufficiency of the condition in Corollary~\ref{cor:surjective}
  can also be proved by using the Cholesky decomposition to compute a
  point in the fiber $\phi_\Delta^{-1}(\Sigma)$ of a matrix
  $\Sigma\in\SSS_{\succeq 0}^m(G)$.  The vertices of a chordal graph
  $G$ can be brought into a perfect elimination ordering, which
  ensures sparsity of the lower-triangular Cholesky factor; see for
  example \cite[Thm.~2.4]{Paulsen1989}.  Suppose the original vertices
  $1,\dots,m$ are already in such an order.  Then $\Sigma=LL^T$ for a
  lower-triangular matrix $L=(l_{ij})$ with $l_{ij}=0$ when $i\not= j$
  and $\{i,j\}$ is not an edge of $G$.  The support of each column of
  $L$ is thus a clique in $G$.  It follows that
  $\Sigma\in\im(\phi_\Delta)$.
  
  The necessity of the chordality condition in
  Corollary~\ref{cor:surjective} also follows from our semi-algebraic
  characterization of $\im(\phi_\Delta)$ when $\Delta$ is the clique
  complex of a chordless cycle; see Section~\ref{sec:chordless-cycles}
  that also gives an example of a matrix not in the image.  \qed
\end{remark}

In the statistical literature, the parametrization $\phi_\Delta$ is
most commonly considered for a simplicial complex $\Delta$ given by
the edges of a graph.  The parametrization for such an edge complex is
surjective only for chordal graphs whose cliques are of cardinality at
most two.  This means that there may not be any cycles.

\begin{corollary}
  \label{cor:forest}
  The edge complex $\Delta$ of a graph $G$ yields a surjective
  parametrization $\phi_\Delta$ of $\mathbb{S}_{\succeq 0}^m(G)$ if
  and only if $G$ is a forest (has no cycles).
\end{corollary}

\section{Submatrices and Schur complements}
\label{sec:sumbatrices-schur-complements}

For a simplicial complex $\Delta$ on $[m]$ and a subset $A\subseteq
[m]$, define the induced subcomplex $\Delta_A = \{ F\in\Delta\::\:
F\subseteq A\}$.

\begin{lemma}
  \label{lem:submatrix}
  Let $\Delta$ be a simplicial complex on $[m]$.  If $\Sigma$ is a
  matrix in the image of $\phi_\Delta$, then all proper principal
  submatrices $\Sigma_{A,A}$, $A \subset [m]$, are in the image of the
  respective induced subcomplex $\phi_{\Delta_A}$.
\end{lemma}
\begin{proof}
  Write $\Sigma = \Gamma(\gamma)\Gamma(\gamma)^T$.  Let
  $\Gamma_{A}(\gamma)$ be the submatrix of $\Gamma(\gamma)$ obtained
  by removing all rows with index not in $A$.  Then $\Sigma_{A,A} =
  \Gamma_{A}(\gamma) \Gamma_{A}(\gamma)^T + \diag(\gamma')$ where
  $\gamma'_i = \sum_{F\in\Delta\setminus\Delta_A} \gamma_{i,F}^2$.
  The matrix $\Gamma_{A}(\gamma) \Gamma_{A}(\gamma)^T$ is in the image
  of $\phi_{\Delta_A}$, and so is the diagonal matrix
  $\diag(\gamma')$.  By convexity (Theorem~\ref{thm:convexity}),
  $\Sigma_{A,A}\in\im(\phi_{\Delta_A})$.
\end{proof}

The converse of the lemma does not hold.  If $\Delta$ is the edge
complex of a chordless cycle, then any matrix in $\mathbb{S}_{\succeq
  0}^m(G)$ has all of its proper principal submatrices in the image of
the corresponding map $\phi_{\Delta_A}$, but
$\im(\phi_G)\subsetneq\mathbb{S}_{\succeq 0}^m(G)$ by
Corollary~\ref{cor:surjective}.

\medskip

For a square matrix $M$ partitioned as
$$
M = \left( 
\begin{array}{cc}
A & B \\
C & D
\end{array}
\right),
$$
the {\em Schur complement} of a non-singular submatrix $D$ in $M$ is
defined as $M/D := A - BD^{-1}C$.  If $M$ is symmetric positive
semidefinite, then so is $M/D$.  If $D$ is further partitioned as
$$
D = \left( 
\begin{array}{cc}
E & F \\
G & H
\end{array}
\right),
$$
and $H$ and $D/H$ are non-singular, then the following {\em
  quotient formula} holds: $M/D = (M/H) / (D/H)$.  Proofs can be found
in textbooks on matrix theory.

For a graph $G = (V,E)$ and a proper subset of vertices $U \subset V$,
define a new graph $G/U$ on vertex set $V\backslash U$ as follows.  A
pair $\{i,j\} \subseteq V\backslash U$ is an edge in $G/U$ if $\{i,j\}
$ is an edge in $G$ or there is a path between $i$ and $j$ in $G$
through vertices in $U$.  For a simplicial complex $\Delta$ on ground
set $V$, define a new simplicial complex $\Delta/ U$ on $V\backslash
U$ where a set $A \subseteq V \backslash U$ forms a face if it is a
face in $\Delta$ or there exists a sequence of distinct elements $u_1,
u_2, \dots, u_k \in U$ and distinct faces $F_1, \dots, F_{k+1} \in
\Delta$ such that $u_i \in F_i \cap F_{i+1}$ and $A =
\bigcup_{i=1}^{k+1} F_i \backslash U$.  If the faces of $\Delta$ form
cliques in $G$, then the faces of $\Delta/U$ form cliques in $G/U$.

\begin{proposition}
  \label{prop:schur}
  Let $\Delta$ be a simplicial complex on $[m]$ and $U\subsetneq [m]$
  a proper subset of nodes.  If $\Sigma$ is in the image of
  $\phi_\Delta$ and $\Sigma_{U,U}$ is non-singular, then the Schur
  complement $\Sigma / \Sigma_{U,U}$ is in the image of
  $\phi_{\Delta/U}$.
\end{proposition}

\begin{proof}
  If $U'$ is a non-empty proper subset of $U$, then by the quotient
  formula we have $\Sigma/\Sigma_{U,U} = (\Sigma / \Sigma_{U', U'}) /
  (\Sigma_{U,U} /\Sigma_{U',U'})$.  Moreover, $\Delta/U = (\Delta/U')
  / (U \backslash U')$ by construction.  Therefore, it suffices to
  prove the assertion when $U$ consists of only one vertex $u$.  We
  call a face $F$ of the complex $\Delta/u := \Delta/\{u\}$ {\em
    original} if $F$ is also a face of $\Delta$ and {\em induced} if
  $F = (F_1 \cup F_2) \backslash \{u\}$ for a pair of distinct faces
  $F_1,F_2$ of $\Delta$ that both contain $u$.  Note that a face can
  be both original and induced.
  
  Let $\Sigma = \Gamma(\gamma) \Gamma(\gamma)^T$ be in the image of
  $\phi_\Delta$.  Define as follows a new matrix
  $\Gamma_{\Delta/u}(\gamma') = [\gamma'_{v,F}]$ whose rows and
  columns are indexed by the vertices and the induced faces of
  $\Delta/u$, respectively.  Fix an arbitrary total order ``$\leq$''
  on the faces of $\Delta$.  For the induced face $F = (F_1 \cup F_2)
  \backslash \{u\}$ given by a pair of faces $F_1 < F_2$ of $\Delta$
  with $u\in F_1\cap F_2$, let
  $$
  \gamma'_{i F} = (\gamma_{i F_1} \gamma_{u F_2} - \gamma_{i F_2}
  \gamma_{u F_1} ) / \sqrt{\sigma_{uu}}.
  $$
  Here, $\gamma_{i F_j}$ is shorthand for $\gamma_{i, F_j}$ and
  $\gamma_{i F_j} = 0$ if $i \notin F_j$.
  
  Let $A = V \backslash \{u\}$. We now show the following:
  \begin{equation}
    \label{eqn:complement}
    \begin{split}
      \Sigma / \Sigma_{u,u} :=&\; \Sigma_{A,A} - \frac{\Sigma_{A,u}
        \Sigma_{u,A}}{\sigma_{uu}} = \; \Gamma_{\Delta_A}(\gamma)
      \Gamma_{\Delta_A}(\gamma)^T + \Gamma_{\Delta/u}(\gamma')
      \Gamma_{\Delta/u}(\gamma')^T,
    \end{split}
  \end{equation}
  where $\Gamma_{\Delta_{A}}(\gamma)$ is the submatrix of
  $\Gamma(\gamma)$ with rows and columns indexed respectively by $A$
  and faces $F \in \Delta$ contained in $A$.  The $ij$-entry on the
  right hand side is
  \begin{align*}
    & \displaystyle \sum_{\begin{subarray} \mbox{F} \in \Delta/u\\ F\:
        \mathrm{original}\end{subarray}  
        } \gamma_{iF} \gamma_{jF} + \sum_{\begin{subarray} \mbox{F}
        \in \Delta/u\\ F\: \mathrm{induced}\end{subarray}}
        \gamma'_{iF} \gamma'_{j,F} \\  
    =&   \sum_{\begin{subarray} \mbox{F} \in \Delta \\
        u\notin F\end{subarray}} \gamma_{iF} \gamma_{jF} +
    \frac{1}{\sigma_{uu}}  \sum_{\begin{subarray}  \mbox{F}_1 < F_2 \in
        \Delta\\ u \in F_1 \cap F_2\end{subarray}} (\gamma_{iF_1}
    \gamma_{uF_2} - \gamma_{iF_2}\gamma_{uF_1}) (\gamma_{jF_1}
    \gamma_{uF_2} - \gamma_{jF_2}\gamma_{uF_1})\\ 
    =&     \sum_{\begin{subarray} \mbox{F} \in \Delta \\
        u\notin F\end{subarray}} \gamma_{iF} \gamma_{jF} +
    \frac{1}{\sigma_{uu}} \sum_{\begin{subarray} \mbox{F}_1 \neq F_2 \in
        \Delta \\ u \in F_1 \cap F_2\end{subarray}} (-
    \gamma_{iF_1}\gamma_{jF_2}\gamma_{uF_1}\gamma_{uF_2} +
    \gamma_{iF_1}\gamma_{jF_1}\gamma_{uF_2}^2)\\ 
    =&   
    \sum_{\begin{subarray} \mbox{F} \in \Delta \\ u\notin F\end{subarray}}
    \gamma_{iF} \gamma_{jF} + \frac{1}{\sigma_{uu}}
    \sum_{\begin{subarray} \mbox{F}_1 , F_2 \in \Delta \\ u \in F_1 \cap
        F_2\end{subarray}} 
    (- \gamma_{iF_1}\gamma_{jF_2}\gamma_{uF_1}\gamma_{uF_2} + \gamma_{iF_1}\gamma_{jF_1}\gamma_{uF_2}^2)\\
    = &  
    \sum_{\begin{subarray} \mbox{F} \in \Delta \\ u\notin F\end{subarray}}
    \gamma_{iF} \gamma_{jF} +  
    \frac{1}{\sigma_{uu}} \sum_{\begin{subarray} \mbox{F}_1 \in \Delta \\ u \in F_1\end{subarray}} \gamma_{iF_1}\gamma_{jF_1}
    \sum_{\begin{subarray} \mbox{F}_2 \in \Delta \\ u \in F_2\end{subarray}}  \gamma_{uF_2}^2
    - \frac{1}{\sigma_{uu}}   \sum_{\begin{subarray} \mbox{F}_1 , F_2 \in \Delta \\ u \in F_1 \cap F_2\end{subarray}}
    \gamma_{iF_1}\gamma_{jF_2}\gamma_{uF_1}\gamma_{uF_2} \\
    =& 
    \sum_{F \in \Delta} \gamma_{iF} \gamma_{jF} - \frac{1}{\sigma_{uu}}   \sum_{\begin{subarray} \mbox{F}_1 , F_2 \in \Delta \\ u \in F_1 \cap F_2\end{subarray}}
    \gamma_{iF_1}\gamma_{jF_2}\gamma_{uF_1}\gamma_{uF_2},
  \end{align*}
  which is equal to the $ij$-entry on the left hand side because
  $\sigma_{ij} = \sum_{F \in \Delta} \gamma_{iF} \gamma_{jF}$.  From
  (\ref{eqn:complement}) and the convexity of $\im(\phi_{\Delta/u})$
  shown in Theorem \ref{thm:convexity}, it follows that Schur
  complement $\Sigma/\Sigma_{u,u}$ is in the image of
  $\phi_{\Delta/u}$.
\end{proof}

The converse of Proposition \ref{prop:schur} cannot hold in general.
Let $G$ be a 3-cycle and consider $\phi_G$ given by the edge complex
of $G$.  Then any Schur complement $\Sigma/ \Sigma_{u,u}$ of $\Sigma
\in \SSS_{\succ 0}^m(G)$ is in $\im(\phi_{G / \{u\}})$, but not every
such matrix $\Sigma$ is in $\im(\phi_G)$.

\section{Chordless cycles}
\label{sec:chordless-cycles}

Let $C_m$ be the chordless $m$-cycle with edges
$\{1,2\},\{2,3\},\dots,\{1,m\}$.  Excluding trivial cases, assume that
$m\ge 3$. Let $\Delta$ be the simplicial complex whose facets are the
edges of $C_m$, and define $\phi_{C_m}=\phi_\Delta$.  For $m \geq 4$, there
is no other simplicial complex $\Delta'$ that gives a parametrization
$\phi_{\Delta'}$ whose image is a full-dimensional subset of
$\SSS_{\succeq 0}^m(C_m)$.  In this section we give a semi-algebraic
description of $\im(\phi_{C_m})$.

We begin with a simple yet important observation.  For a symmetric
matrix $\Sigma\in\SSS^m$ and two distinct indices $i,j\in [m]$, define
$\Sigma^{(ij)}$ to be the symmetric matrix obtained by negating the
$(i,j)$ and $(j,i)$ entries of $\Sigma$.

\begin{lemma}
  \label{lem:negating}
  Suppose $\Sigma\in\SSS_{\succeq 0}^m$ is a positive semidefinite
  matrix, and $\Delta$ is the edge complex of a graph.  If $i\not= j$
  and $\{i,j\}\in \Delta$, then $\Sigma$ is in the image of
  $\phi_\Delta$ if and only if $\Sigma^{(ij)}$ is in
  $\im(\phi_\Delta)$.
\end{lemma}
\begin{proof}
  If $\Sigma=\phi_\Delta(\gamma)$, then
  $\Sigma^{(ij)}=\phi_\Delta(\bar\gamma)$, where $\bar\gamma$ is
  identical to $\gamma$ except for the single entry
  $\gamma_{ij}=\gamma_{i,\{i,j\}}$ that is replaced by its negative
  $-\gamma_{ij}$.  The entry $\gamma_{ji}$ remains unchanged.
\end{proof}

Note that Lemma~\ref{lem:negating} immediately yields the necessary
condition stated in (\ref{eq:det-R-12}) in
Example~\ref{ex:three-cycle} about $C_3=K_3$, the complete graph on 3
nodes.  The lemma can also be used to give an explicit example of a
matrix not in the image of the parametrization for $C_m$, $m\ge 3$.

\begin{example}
  \label{ex:counter-Cm}
  For $m\ge 3$, define the symmetric $m\times m$ matrix
  \[
  \Sigma(\rho)=
  \begin{pmatrix}
    1 & \tfrac{1}{2} & & & & \tfrac{1}{2}\rho\\
    \tfrac{1}{2} & 1 & \tfrac{1}{2} &  &  &\\
    & \ddots & \ddots & \ddots & \ddots \\
    &  &  & \tfrac{1}{2} & 1 & \tfrac{1}{2}\\
    \tfrac{1}{2}\rho &  &  &   & \tfrac{1}{2}& 1
  \end{pmatrix},
  \]
  where all omitted entries are zero such that
  $\Sigma(\rho)\in\SSS^m(C_m)$.  Omitting the $m$-th row and column of
  $\Sigma(\rho)$ yields a positive definite tridiagonal matrix.
  Hence, $\Sigma(\rho)$ is in the graphical cone $\SSS_{\succeq
    0}^m(C_m)$ if and only if $\det(\Sigma(\rho))\ge 0$.  Using Laplace
  expansions and the recursive formula for the determinant of a
  tridiagonal matrix, one can show that
  \[
  \det \Sigma(\rho) = 
  \begin{cases}
    \frac{1}{2^m}\big[ m+1 - (m-1)\rho\big](1+\rho) &\text{if}\; m
    \,\text{is odd},\\
    \frac{1}{2^m}\big[ m+1 + (m-1)\rho\big](1-\rho) &\text{if}\; m
    \,\text{is even}.
  \end{cases}
  \]
  If $m$ is odd, choose $\rho\in (1,1+2/(m-1))$.  If $m$ is even,
  choose $\rho\in(-1-2/(m-1),-1)$.  Then the determinant of
  $\Sigma(\rho)$ is positive but the determinant of $\Sigma(-\rho)$ is
  negative.  Since $\Sigma(\rho)^{(1m)}=\Sigma(-\rho)$, it follows
  from Lemma~\ref{lem:negating} that $\Sigma(\rho)$ is in
  $\SSS_{\succeq 0}^m(C_m)$ but not in the image of $\phi_{C_m}$.
  \qed
\end{example}

We now state the main result of this section, a semi-algebraic
description of the image of $\phi_{C_m}$.  Let $\mathcal{M}(C_m)$ be
the collection of all (not necessarily maximal) matchings of $C_m$.  A
matching is any set $M$ of pairwise disjoint edges, that is,
$\{i,j\},\{k,l\}\in M$ implies that $\{i,j\}\cap\{k,l\}=\emptyset$.
For $M\in\mathcal{M}(C_m)$, we write $[m]\setminus M$ to denote the
set of nodes not incident to any edge in $M$.  If
$\Sigma=(\sigma_{ij})$ is a matrix in $\SSS_{\succeq 0}^m(C_m)$ then
its determinant can be expanded as
\begin{align}
  \label{eq:detSigma-expansion}
  \det(\Sigma) &= (-1)^{m+1}\cdot 2 \prod_{i=1}^m \sigma_{i,i+1}\; +\;
  \sum_{M\in\mathcal{M}(C_m)} (-1)^{|M|}
  \prod_{\{i,j\}\in M} \sigma_{ij}^2 \prod_{i \in\, [m]\setminus
    M} \sigma_{ii};
\end{align}
compare \cite[\S1.4, eqn.~(1.42)]{Cvetkovic1995}.
In the first term, which corresponds to the entire cycle $C_m$, the indices
are read modulo $m$ such that $\sigma_{m,m+1}\equiv \sigma_{1m}$.  The
following theorem is the main result of this section.

\begin{theorem}
  \label{thm:semi-alg-cycle}
  A matrix $\Sigma=(\sigma_{ij})\in \SSS_{\succeq 0}^m(C_m)$ is in the
  image of $\phi_{C_m}$ if and only if
  \[
  \sum_{M\in\mathcal{M}(C_m)} (-1)^{|M|}
  \prod_{\{i,j\}\in M} \sigma_{ij}^2 \prod_{i \in\, [m]\setminus
    M} \sigma_{ii} \quad \ge \quad 2 \prod_{i=1}^m
  |\sigma_{i,i+1}|.
  \]
\end{theorem}
\begin{proof}  {\em (Necessity)}
  A matrix $\Sigma\in\im(\phi_{C_m})$ is positive semidefinite and
  thus has a non-negative determinant.  Pick any edge in $C_m$, say
  $\{1,2\}$.  Then, by Lemma~\ref{lem:negating}, $\Sigma^{(12)}$ is
  positive semidefinite.  Hence,
  \begin{multline*}
    \min\left\{\det(\Sigma),\det(\Sigma^{(12)})\right\} =\\
    -2
    \prod_{i=1}^m |\sigma_{i,i+1}|\; +\;
    \sum_{M\in\mathcal{M}(C_m)} (-1)^{|M|}
    \prod_{\{i,j\}\in M} \sigma_{ij}^2 \prod_{i \in\,
      [m]\setminus M} \sigma_{ii} \quad \ge \quad 0.
  \end{multline*}
  
  {\em (Sufficiency)} We need to show that under the assumed condition
  on $\Sigma\in\SSS_{\succeq 0}(C_m)$, the equation system
  $\phi_{C_m}(\gamma)=\Sigma$ has a feasible solution $(\gamma)$ in
  $(\mathbb{R})^V \times (\mathbb{R}^2)^E$.  Our proof will show that
  $\gamma_{i,\{i\}}$ does not play an important role and can simply be
  set to zero.  Since $\im(\phi_{C_m})$ is closed and full-dimensional
  in $\SSS_{\succeq 0}(G)$, it suffices to show that a dense (Zariski
  open) subset of points $\Sigma$ satisfying the necessity condition
  is contained in $\im(\phi_{C_m})$.  We show that for positive
  definite $\Sigma$, any complex solution $\gamma$ (with
  $\gamma_{i,\{i\}} = 0$ for all $i = 1,\dots,m$) is in fact real and
  thus feasible (Lemma~\ref{lem:quartic}).  The proof is completed by
  demonstrating the existence of complex solutions for generic
  $\Sigma$ (Lemma~\ref{lem:bernstein}).
 \end{proof}

The following is an immediate consequence of  Theorem \ref{thm:semi-alg-cycle}. 

\begin{corollary}
\label{cor:suff}
For a positive semidefinite matrix $\Sigma = (\sigma_{ij})$ in $\mathbb{S}_{\succeq 0}(C_m)$ the following are equivalent:
\begin{enumerate}
\item $\Sigma$ is in the image of $\phi_{C_m}$.
\item $\Sigma^{(ij)}$ is positive semidefinite for all edges $ij \in C_m$.
\item $\Sigma^{(ij)}$ is positive semidefinite for some edge $ij \in C_m$.
\end{enumerate}
\end{corollary}

\begin{proof}
By Lemma \ref{lem:negating}, we have $(1)$ implies $(2)$.  It is obvious that $(2)$ implies $(3)$.  If both $\Sigma$ and $\Sigma^{(ij)}$ have non-negative determinants for some edge $ij \in C_m$, then the inequality in Theorem \ref{thm:semi-alg-cycle} is satisfied, so we have $(3)$ implies $(1)$.
\end{proof}

The semi-algebraic condition from Theorem~\ref{thm:semi-alg-cycle} is
easily verified, and we can use it to compute the spherical volume of
the cone $\im(\phi_{C_m})$ by Monte Carlo integration.
Table~\ref{tab:volumes} shows which fraction of the cone
$\mathbb{S}_{\succeq 0}(C_m)$ is covered by the image of $\phi_{C_m}$,
when quantifying this by the ratio of the spherical volumes of the two
cones.  The rounded ratios were computed by simulating 100,000
matrices in $\mathbb{S}_{\succeq 0}(C_m)$.  In terms of the ratio of
spherical volumes, the difference between $\im(\phi_{C_m})$ and
$\mathbb{S}_{\succeq 0}(C_m)$ is largest for $m= 3$ and becomes rather
minor for $m=6,7$.

\begin{table}[t]
  \centering
  \caption{Spherical volume of the image of $\phi_{C_m}$ as a fraction
    of the spherical volume of the cone $\mathbb{S}_{\succeq 0}(C_m)$.} 
  \label{tab:volumes}
  \begin{tabular}{cccccc}
    \hline    \hline
    $m$ & 3 & 4 & 5 & 6 & 7\\
    Vol & 0.78 & 0.90 & 0.95 & 0.98 & 0.99\\
    \hline
  \end{tabular}
\end{table}

\medskip

The remainder of the section is devoted to the details of the proof of
the sufficiency of the condition in Theorem~\ref{thm:semi-alg-cycle}.
As mentioned above, our approach is to study the equation system
$\phi_{\Delta}(\gamma)=\Sigma$ for a given matrix
$\Sigma=(\sigma_{ij})\in\SSS_{\succeq 0}(G)$, which is treated as a
parameter to the system.  For $\Delta= E(C_m)$, the equations take the
form:
\begin{subequations}
\begin{align}
  \label{eq:cycle-eqns1}
  \gamma_{i,\{i\}}^2+\gamma_{i,i-1}^2+\gamma_{i,i+1}^2&=\sigma_{ii},  &i=1,\dots,m,\\
  \label{eq:cycle-eqns2}
  \gamma_{i-1,i}\gamma_{i,i-1}&=\sigma_{i-1,i}, &i=1,\dots,m.
\end{align}
\end{subequations}
where $\gamma_{i, j}$ denotes $\gamma_{i,\{i,j\}}$ and indices are read modulo $m$ such that $0\equiv m$ and $1\equiv m+1$.
The $2m$ equations in (\ref{eq:cycle-eqns1}) and (\ref{eq:cycle-eqns2})
involve $3m$ unknowns and have an $m$-dimensional solution set in
$\mathbb{C}^{3m}$.  In the sequel we simply omit the unknowns $\gamma_{i,\{i\}}$ from the
system.  In other words, we study the $2m$ equations in $2m$ unknowns:
\begin{subequations}
\begin{align}
  \label{eq:cycle-eqns1-no-diag}
  \gamma_{i,i-1}^2+\gamma_{i,i+1}^2&=\sigma_{ii},  &i=1,\dots,m,\\
  \label{eq:cycle-eqns2-no-diag}
  \gamma_{i-1,i}\gamma_{i,i-1}&=\sigma_{i-1,i}, &i=1,\dots,m.
\end{align}
\end{subequations}
%


We will show that  system (\ref{eq:cycle-eqns1-no-diag})-(\ref{eq:cycle-eqns2-no-diag}) has a real solution if $\Sigma$ satisfies the condition from Theorem~\ref{thm:semi-alg-cycle}, which will imply that the system (\ref{eq:cycle-eqns1})-(\ref{eq:cycle-eqns1}) has a real solution, too.
Somewhat surprisingly it suffices to argue that they
have a complex solution, as is made precise in the next lemma.

\begin{lemma}
  \label{lem:quartic}
  If a positive definite matrix $\Sigma\in\SSS_{\succeq 0}(C_m)$
  satisfies the necessary condition from
  Theorem~\ref{thm:semi-alg-cycle}, then all complex solutions to the
  equations~(\ref{eq:cycle-eqns1-no-diag})-(\ref{eq:cycle-eqns2-no-diag})
  are in fact real.
\end{lemma}
\begin{proof}
  Suppose the vector $\gamma\in(\mathbb{C}^2)^E$ provides a
  solution to
  (\ref{eq:cycle-eqns1-no-diag})-(\ref{eq:cycle-eqns2-no-diag}).  Fill
  the $2m$ unknowns in a matrix $\Gamma(\gamma)\in\mathbb{C}^{V\times
  E}\simeq \mathbb{C}^{m\times m}$
  as in (\ref{eq:def-Gamma-gamma}).  Note that we are setting $\gamma_{i,\{i\}} = 0$ for all $i = 1,\dots,m$, so the columns in $\Gamma(\gamma)$ corresponding to the singleton faces $\{1\}, \dots, \{m\}$ are zero and can be omitted.  
  Augment $\Gamma(\gamma)$ to an $(m+1)\times
  m$ matrix $\Gamma_{(12)}$ by adding the vector $(\gamma_{12},0,\dots,0)$
  as a first row.  Based on the equations
  (\ref{eq:cycle-eqns1-no-diag})-(\ref{eq:cycle-eqns2-no-diag}),
  \begin{equation}
    \label{eq:quartic-matrix}
    \Gamma_{(12)}\Gamma_{(12)}^T = 
    \begin{pmatrix}
      \gamma_{12}^2 & \gamma_{12}^2 & \sigma_{12} & \\
      \gamma_{12}^2 & \sigma_{11} & \sigma_{12} &  & \sigma_{1m}\\
      \sigma_{12} & \sigma_{12} & \sigma_{22} & \sigma_{23} &\\
      && \sigma_{23} & \sigma_{33} & \ddots\\
      &&&\ddots & \ddots & \sigma_{m-1,m}\\
      & \sigma_{1m} &&  & \sigma_{m-1,m} & \sigma_{mm}
    \end{pmatrix}
  \end{equation}
  with blank entries being zero.  The $(m+1)\times (m+1)$ matrix
  $\Gamma_{(12)}\Gamma_{(12)}^T$ has rank at most $m$, and thus its
  determinant vanishes.  Therefore, $\gamma_{12}$ has to satisfy the
  quartic equation
  \begin{equation}
    \label{eq:quartic}
    \det\left(\Gamma_{(12)}\Gamma_{(12)}^T\right) = 
    a\gamma_{12}^4 + b\gamma_{12}^2+c = 0.
  \end{equation}
  By expanding the determinant of (\ref{eq:quartic-matrix}) along the first
  row (or column), the coefficients in (\ref{eq:quartic}) are
  found to be:
  \begin{align*}
    a &=-\det\left(\Sigma_{[m]\setminus\{1\},[m]\setminus\{1\}}\right),\\
    b &=  \det(\Sigma) +
    2\sigma_{12}^2
    \det\left(\Sigma_{[m]\setminus\{1,2\},[m]\setminus\{1,2\}}\right) 
    +(-1)^m\cdot 2\prod_{i=1}^m \sigma_{i-1,i},\\ 
    c&= -\sigma_{12}^2
    \det\left(\Sigma_{[m]\setminus\{2\},[m]\setminus\{2\}}\right).
  \end{align*}
  Note that the third term in $b$ cancels out a term in $\det(\Sigma)$;
  recall~(\ref{eq:detSigma-expansion}).  It remains to argue that under the
  assumption on $\Sigma$, the quartic in (\ref{eq:quartic}) has only real
  solutions.  Define $\Sigma^{(12)}$ by negating $\sigma_{12}$ as in
  Lemma~\ref{lem:negating}.  Then the assumption on $\Sigma$ implies that
  both $\det(\Sigma)$ and $\det(\Sigma^{(12)})$ are positive.
 
  First, we claim that the discriminant $b^2-4ac$ is positive because
  \begin{equation}
    \label{eq:discr}
    b^2-4ac = \det(\Sigma)\det(\Sigma^{(12)}).
  \end{equation}
  Since
  \begin{equation}
    \label{eq:sigma-sigmaneg}
  \det(\Sigma)+(-1)^m\cdot 2\prod_{i=1}^m \sigma_{i-1,i} = 
  \det(\Sigma^{(12)})-(-1)^m\cdot 2\prod_{i=1}^m \sigma_{i-1,i} 
  \end{equation}
  we can write
  \begin{multline*}
    b^2 = \left[\det(\Sigma) + 2\sigma_{12}^2
      \det\left(\Sigma_{[m]\setminus\{1,2\},[m]\setminus\{1,2\}}\right)
      +(-1)^m\cdot 2\prod_{i=1}^m \sigma_{i-1,i}\right]\times\\
    \left[\det(\Sigma^{(12)}) + 2\sigma_{12}^2
      \det\left(\Sigma_{[m]\setminus\{1,2\},[m]\setminus\{1,2\}}\right)
      -(-1)^m\cdot 2\prod_{i=1}^m \sigma_{i-1,i} \right].
  \end{multline*}
  Multiplying out the product, and using~(\ref{eq:sigma-sigmaneg}) once
  more, yields that
  \begin{multline}
     \label{eq:b2-simplified}
     b^2 = 
     \det(\Sigma)\det(\Sigma^{(12)}) +
     4\sigma_{12}^4
     \det(\Sigma_{[m]\setminus\{1,2\},[m]\setminus\{1,2\}})^2
     + 4\prod_{i=1}^m \sigma_{i-1,i}^2+\\
     2\sigma_{12}^2\det(\Sigma_{[m]\setminus\{1,2\},[m]\setminus\{1,2\}})
     \left(\det(\Sigma)+\det(\Sigma^{(12)}) \right).
  \end{multline}
  Expansion of determinants shows that
  \begin{multline*}
  \det(\Sigma) = 
  \sigma_{11} \det(\Sigma_{[m]\setminus \{1\},[m]\setminus \{1\}})
  -\sigma_{12}^2\det(\Sigma_{[m]\setminus \{1,2\},[m]\setminus \{1,2\}})\\
  -\sigma_{1m}^2\det(\Sigma_{[m-1]\setminus \{1\},[m-1]\setminus \{1\}})
  -(-1)^m\cdot 2\prod_{i=1}^m \sigma_{i-1,i}.
  \end{multline*}
  Hence,
  \begin{multline}
    \label{eq:expansion1}
    \det(\Sigma)+\det(\Sigma^{(12)}) = 
    2\sigma_{11} \det(\Sigma_{[m]\setminus \{1\},[m]\setminus \{1\}})\\
    -2\sigma_{12}^2\det(\Sigma_{[m]\setminus \{1,2\},[m]\setminus \{1,2\}})
    -2\sigma_{1m}^2\det(\Sigma_{[m-1]\setminus \{1\},[m-1]\setminus \{1\}}).
  \end{multline}
  Moreover, we find from another expansion that
  \begin{multline}
    \label{eq:expansion2}
    \det(\Sigma_{[m]\setminus\{2\},[m]\setminus\{2\}})= \\
    \sigma_{11}\det(\Sigma_{[m]\setminus\{1,2\},[m]\setminus\{1,2\}})
    +\sigma_{1m}^2\det(\Sigma_{[m-1]\setminus\{1,2\},[m-1]\setminus\{1,2\}}).   
  \end{multline}
  Combining (\ref{eq:b2-simplified}), (\ref{eq:expansion1}) and
  (\ref{eq:expansion2}), we obtain that our claim (\ref{eq:discr}) holds if
  \begin{multline}
    \label{eq:tridiag-property-for-discriminant}
    \sigma_{12}^2\sigma_{1m}^2\det(\Sigma_{[m]\setminus\{1\},[m]\setminus\{1\}})
    \det(\Sigma_{[m-1]\setminus\{1,2\},[m-1]\setminus\{1,2\}}) =\\
    \sigma_{12}^2\sigma_{1m}^2
    \det(\Sigma_{
      [m]\setminus\{1,2\},[m]\setminus\{1,2\}})
    \det(\Sigma_{[m-1]\setminus \{1\},[m-1]\setminus \{1\}}) -
    \prod_{i=1}^m \sigma_{i-1,i}^2 .
  \end{multline}
  However, all determinants appearing in
  (\ref{eq:tridiag-property-for-discriminant}) are determinants of
  tridiagonal matrices and thus
  (\ref{eq:tridiag-property-for-discriminant}) can be shown to hold by an
  induction on $m$.
  
  As just established, $b^2-4ac>0$.  Moreover, since we assume that both
  $\det(\Sigma)$ and $\det(\Sigma^{(12)})$ are positive it holds that
  $b>0$; see e.g.~(\ref{eq:b2-simplified}).  In addition, $a$ and $c$ are
  both negative, and it follows from the usual formula for the solutions of
  a quadratic equation that all four solutions to the quartic equation in
  (\ref{eq:quartic}) are real.  Hence, $\gamma_{12}$ is real and, by
  symmetry, the same is true for all other components of a solution to
  (\ref{eq:cycle-eqns1-no-diag})-(\ref{eq:cycle-eqns2-no-diag}).
\end{proof}

In the final step that completes the proof of sufficiency of the
condition in Theorem~\ref{thm:semi-alg-cycle}, we show that for a
generic choice of $\Sigma$, the
equations~(\ref{eq:cycle-eqns1-no-diag})-(\ref{eq:cycle-eqns2-no-diag})
indeed have at least one complex solution.  We are able to restrict
attention to generic choices of the entries of $\Sigma$ because
$\im(\phi_{C_m})$ is closed and full-dimensional in $\SSS_{\succeq
  0}(C_m)$.  In particular, we may assume that $\sigma_{i-1,i}\not=0$
for all $i\in[m]$.  This implies that a solution of
(\ref{eq:cycle-eqns1-no-diag})-(\ref{eq:cycle-eqns2-no-diag}) has all
components non-zero.  Moreover, the solution set of
(\ref{eq:cycle-eqns1-no-diag})-(\ref{eq:cycle-eqns2-no-diag}) is
identical, up to sign, to that of the rational system
\[
\gamma_{i,i+1}^2+\frac{\sigma_{i-1,i}^2}{\gamma_{i-1,i}^2}=\sigma_{ii},\quad
i=1,\dots,m.
\]
Or simpler yet, setting $x_i=\gamma_{i-1,i}$ (in particular, $x_0\equiv
x_m$), the solution set corresponds exactly to the solution set of the polynomial
system
\begin{equation}
  \label{eq:equations-torus}
  x_{i+1}^2x_{i}^2-\sigma_{ii}x_i^2+\sigma_{i-1,i}^2 = 0,
  \quad i=1,\dots,m,
\end{equation}
in the torus $(\mathbb{C}^*)^m=(\mathbb{C}\setminus\{0\})^m$.

\begin{lemma}
  \label{lem:bernstein}
  For generic choices of the coefficients $\sigma_{ij}$, the equations
  (\ref{eq:equations-torus}) have $2^{m+1}$ solutions in the torus
  $(\mathbb{C}^*)^m$.
\end{lemma}

\begin{proof}
  We can rewrite the equation system in (\ref{eq:equations-torus}) as
  $$
  x_{i+1}^2 = \sigma_{ii} - \frac{\sigma_{i-1,i}^2}{x_i^2},\quad i
  = 1,\dots,m.
  $$
  Solving the equation for $i=1$ for $x_2$, plugging the result
  into the equation for $i=2$ and solving for $x_3$, and continuing on
  in this fashion with all of the first $m-1$ equations, we can write
  for each $i = 2,\dots,m$,
  $$
  x_i^2 = \frac{a_i x_1^2 + b_i}{c_i x_1^2 + d_i},
  $$
  where $a_i,b_i,c_i,d_i$ are polynomials in the coefficients
  $\sigma_{kk},\sigma_{k-1,k}^2$.  From the last equation,
  we then obtain that
  \begin{equation}
    \label{eq:x1}
    x_1^2 = \frac{a x_1^2 + b}{c x_1^2 + d},
  \end{equation}
  where $a,b,c,d$ are again polynomials in the coefficients
  $\sigma_{kk},\sigma_{k-1,k}^2$.
  
  Now specialize to the case of all $\sigma_{kk} = 1$ and all
  $\sigma_{k-1,k}^2 = -1$.  Then we get 
  \[
  x_2^2 = \frac{x_1^2 + 1}{x_1^2}, \quad 
  x_3^2 = \frac{2x_1^2 + 1}{x_1^2+1}, \dots,
  \]
  and finally equation~(\ref{eq:x1}) becomes
  \begin{equation}
    \label{eq:x1-special}
    x_1^2 = \frac{F_{m+1} x_1^2 + F_{m}}{F_{m}
      x_1^2+F_{m-1}},
  \end{equation}
  where $F_m$ is the $m^\text{th}$ term of the Fibonacci sequence
  $1,1,2,3,\dots$.  Clearing denominators, (\ref{eq:x1-special})
  simplifies to $x_1^4 - x_1^2 - 1 = 0$.  This equation is obviously
  not identical to $x^4 =0$ or $x^2=0$, and its discriminant is not
  identically zero.  The coefficients in (\ref{eq:x1}) being
  polynomial, it follows that for generic complex numbers
  $\sigma_{kk},\sigma_{k-1,k}^2$, the equation in (\ref{eq:x1}) has two
  distinct non-zero solutions up to sign, and the system
  (\ref{eq:equations-torus}) has $2^{m+1}$ solutions.
\end{proof}

\begin{lemma}
  For any positive definite matrix $\Sigma = (\sigma_{ij})$ in image of
  $\phi_{C_m}$, the fiber $\{\gamma : \phi_{C_m}(\gamma) = \Sigma \mbox{
    and } \gamma_{i,\{i\}} = 0 \mbox{ for all } i = 1,\dots,m \}$ consists
  of exactly $2^{m+1}$ elements, or two elements up to sign.
\end{lemma}

\begin{proof}
  First consider the case when all $\sigma_{i,i+1}$ are non-zero.
  Since the elements of the fiber are in bijection with the solutions
  of (\ref{eq:equations-torus}) under setting $x_{i} =
  \gamma_{i-1,i}$, it suffices to show that (\ref{eq:equations-torus})
  has $2^{m+1}$ solutions.  For generic $\Sigma$, this was done in
  Lemma \ref{lem:bernstein}.  For an arbitrary $\Sigma$, let
  $\Sigma_n$ be a sequence of generic matrices in $\im(\phi_{C_m})$
  that converges to $\Sigma$.  Each of the $2^{m+1}$ solutions for
  $\Sigma_n$ can be expressed in terms of radicals using
  (\ref{eq:quartic}) and (\ref{eq:equations-torus}), and each of these
  $2^{m+1}$ sequences converge to solutions of
  (\ref{eq:equations-torus}) for $\Sigma$ by continuity.  Moreover the
  limit points are distinct because the discriminant $b^2-4ac$ in
  (\ref{eq:quartic}) is positive for $\Sigma$ as shown the proof of
  Lemma \ref{lem:quartic} and all $x_i$ in a solution must be
  non-zero.
 
  Now suppose $\sigma_{12}=0$.  In this case Lemma \ref{lem:quartic}
  still holds but Lemma \ref{lem:bernstein} does not apply.  We will
  now show that the system
  (\ref{eq:cycle-eqns1-no-diag})-(\ref{eq:cycle-eqns2-no-diag}) still
  has $2^{m+1}$ complex solutions.  Since $\sigma_{12} = 0$, we have
  either $\gamma_{12} = 0$ or $\gamma_{21} = 0$.  If $\gamma_{21} =
  0$, then we can determine $\gamma_{23}$, $\gamma_{32}$,
  $\gamma_{34}$, $\gamma_{43}$, $\dots$, $\gamma_{1,m}, \gamma_{12}$
  in that order along the cycle, using (\ref{eq:cycle-eqns1-no-diag})
  and (\ref{eq:cycle-eqns2-no-diag}) alternatingly.  These equations
  show that the sequence of $\gamma_{ij}$ obtained this way is unique
  up to sign unless we get $\gamma_{i,i+1} = 0$ for some $i$.
  However, if $\gamma_{i,i+1} = 0$, then the principal submatrix of
  $\Sigma$ indexed by $\{2, 3, \dots, i\}$ would be equal to
  $\Gamma(\gamma) \Gamma(\gamma)^T$ where $\Gamma(\gamma)$ is defined
  as in the introduction for the subgraph on vertices $2, 3, \dots, i$
  and edges $\{2,3\}, \dots, \{i-1,i\}$. Then $ \Gamma(\gamma)$ would
  have more rows than non-zero columns since $\gamma_{i,\{i\}} = 0$
  for all $i$, so $\Gamma(\gamma) \Gamma(\gamma)^T$ would not have
  full rank, contradicting the hypothesis that $\Sigma$ is positive
  definite.  Hence $\gamma_{i,i+1} \neq 0$ for all $i$, and setting
  $\gamma_{21} = 0$ determines all other $\gamma_{ij}$ up to sign.
  There are two choices of signs for each pair $\gamma_{i,i+1}$ and $
  \gamma_{i+1,i}$, even if $\sigma_{i,i+1} = 0$, so there are $2^m$
  solutions with $\gamma_{21} = 0$.
  
  By symmetry, if $\gamma_{12} = 0$, then we can determine
  $\gamma_{1,m}, \gamma_{m,1}, \gamma_{m,m-1}, \dots, \gamma_{23},
  \gamma_{21}$ in that order (going around the cycle in the other
  direction, with $\gamma_{i,i-1} \neq 0$ in this case), so there are
  $2^m$ solutions when $\gamma_{12} = 0$ also.  We cannot have both
  $\gamma_{21} = 0$ and $\gamma_{12} = 0$ because that would imply
  that $\Gamma(\gamma) \Gamma(\gamma)^T$ is singular.  So there are
  $2^{m+1}$ distinct solutions (two solutions up to sign) to the
  system (\ref{eq:cycle-eqns1-no-diag})-(\ref{eq:cycle-eqns2-no-diag})
  when $\sigma_{12} = 0$.  By symmetry, there are $2^{m+1}$ solutions
  for every $\Sigma \in \im(\phi_{C_m})$.
\end{proof}

\section{Statistical models and bipartite acyclic digraphs}
\label{sec:statistical-models}

In probability theory, positive semidefinite matrices arise as
covariance matrices of random vectors.  When the random vector
$Y=(Y_1,\dots,Y_m)$ is Gaussian (has a multivariate normal
distribution) with covariance matrix $\Sigma=(\sigma_{ij})$, then
$\sigma_{ij}=0$ is equivalent to the stochastic independence of the
two random variables $Y_i$ and $Y_j$.  Hence, the convex cone
$\SSS_{\succeq 0}^m(G)$ of positive semidefinite matrices with zeros
at the non-edges of a graph $G$ collects all covariance matrices for
which the components of $Y$ exhibit a pattern of independences.

For a simplicial complex $\Delta$ on $[m]$ with underlying graph $G$,
the map $\phi_\Delta$ traces out a full-dimensional subset of
$\SSS_{\succeq 0}^m(G)$.  This subset arises quite naturally for
random vectors whose components are linear combinations of a set of
independent random variables.  We review this construction next.

Let $\Delta_2$ be the set of all faces in $\Delta$ that have
cardinality at least two.  Introduce the random variables
$\varepsilon_i$, $i\in[m]$, and $H_F$, $F\in\Delta_2$.  Suppose the
random variables $H_F$ are mutually independent with a standard normal
distribution, denoted $\mathcal{N}(0,1)$.  Suppose further that the
$\varepsilon_i$ are mutually independent, independent of the $H_F$,
and distributed as
$\varepsilon_i\sim\mathcal{N}(0,\gamma_{i,\{i\}}^2)$ where
$\gamma_{i,\{i\}}^2$ is the variance.  Define new random variables
$Y_1,\dots,Y_m$ as linear combinations:
\begin{equation}
  \label{eq:def-rv-Y}
  Y_i = \sum_{F\in\Delta_2:i\in F} \gamma_{i,F} H_F + \varepsilon_i,
\qquad i\in[m].
\end{equation}

\begin{proposition}
  \label{prop:hidden-margin}
  The random vector $Y=(Y_1,\dots,Y_m)$ defined by (\ref{eq:def-rv-Y})
  has the positive semidefinite matrix $\phi_\Delta(\gamma)$ as
  covariance matrix.
\end{proposition}
\begin{proof}
  Write $I$ for the identity matrix (of the appropriate
  size).  Let $\varepsilon=(\varepsilon_1,\dots,\varepsilon_m)$ and
  $H=(H_F : F\in\Delta_2)$.  The concatenation $(\varepsilon,H)$
  is a random vector with the diagonal $|\Delta|\times |\Delta|$
  covariance matrix
  \begin{equation}
    \label{eq:Omega}
    \Omega =
    \begin{pmatrix}
      \diag(\gamma_{i,\{i\}}^2) & 0 \\
      0 & I  
    \end{pmatrix}.
  \end{equation}
  Let $\Gamma_2(\gamma)$ be the submatrix of $\Gamma(\gamma)$
  obtained by retaining only the columns corresponding to faces in
  $\Delta_2$; recall (\ref{eq:def-Gamma-gamma}). Define the
  $|\Delta|\times |\Delta|$ matrix
  \begin{equation}
    \label{eq:digraph-covmatrices}
    \Lambda =
    \begin{pmatrix}
      I & -\Gamma_2(\gamma) \\
      0 & I  
    \end{pmatrix}.
  \end{equation}
  Multiplying $\Lambda$ with $(Y,H)^T$ gives the vector
  $(\varepsilon,H)^T$.  By standard results about linear combinations
  of random variables, it follows that the Gaussian random vector
  $(Y,H)$ has covariance matrix $\Lambda^{-1}\Omega\Lambda^{-T}$.  The
  covariance matrix of $Y$ alone is the principal submatrix given by
  the first $m$ rows and columns of $\Lambda^{-1}\Omega\Lambda^{-T}$.
  The inverse $\Lambda^{-1}$ is obtained by negating the upper right
  block, which becomes simply $\Gamma_2(\gamma)$.  It follows that, as
  claimed,
  \[
  \diag(\gamma_{i,\{i\}}^2) + \Gamma_2(\gamma)\Gamma_2(\gamma)^T = 
  \Gamma(\gamma)\Gamma(\gamma)^T =
  \phi_\Delta(\gamma).  
  \qedhere
  \]
\end{proof}

In the field of graphical statistical modelling, it is customary to
visualize an equation system such as (\ref{eq:def-rv-Y}) by means of
an acyclic digraph; see for instance \cite[Chap.~3]{Drton:2009}.
Here, we draw the digraph $D_\Delta$ that has vertex set $\Delta$ and
the edges $F\to \{i\}$ for all pairs of an index $i\in [m]$ and a face
$F\in\Delta_2$ with $i\in F$.  Note that $D_\Delta$ is bipartite with
respect to the partitioning $\Delta=\Delta_1\cup\Delta_2$, where
$\Delta_1=\{ \{i\} : i\in [m]\}$ are the singleton faces and
$\Delta_2$ was defined above.  See Figure \ref{fig:bipartite} for an example.

\begin{figure}[t]
  \centering
  \vspace{0.3cm}
  \scalebox{0.5}{\includegraphics{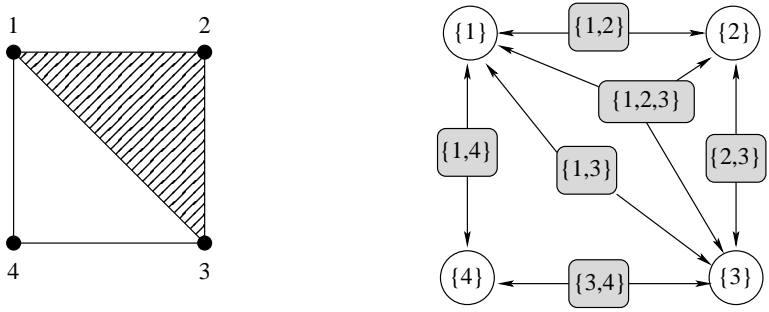}}
  \caption{A simplicial complex (left) and the acyclic bipartite
    digraph corresponding to it (right).}
    \label{fig:bipartite}
\end{figure}

In the setup of Proposition~\ref{prop:hidden-margin}, the random
variables $Y_i$ are functions of the hidden variables $H_F$, up to the
noise given by the $\varepsilon_i$.  There is a dual construction in
which the hidden variables are functions of the observed variables.
Suppose that the random variables $\bar Y_i$ are mutually independent
and normally distributed as $\bar
Y_i\sim\mathcal{N}(0,1/\gamma_{i,\{i\}}^2)$ with
$\gamma_{i,\{i\}}^2\not=0$.  Suppose further that $\nu_F$,
$F\in\Delta_2$, are mutually independent $\mathcal{N}(0,1)$ random
variables that are also independent of the $\bar Y_i$.  Define 
random variables $\bar H_F$ as linear combinations:
\begin{equation}
  \label{eq:def-rv-H}
  \bar H_F = \sum_{i\in [m]: i\in F} \gamma_{i,F} \bar Y_i + \nu_F, \qquad
  F\in\Delta_2.
\end{equation}
Note that this equation system is associated with the bipartite
acyclic digraph obtained by reversing the direction of all edges in
$D_\Delta$.

\begin{proposition}
  \label{prop:hidden-condition}
  If the random vector $\bar H=(\bar H_F: F\in\Delta_2)$ is
  defined by (\ref{eq:def-rv-H}), then the positive definite matrix
  $\phi_\Delta(\gamma)$ is the inverse of the covariance matrix of the
  conditional distribution of $\bar Y$ given $\bar H$.
\end{proposition}
\begin{proof}
  Concatenating $\bar Y$ and $\bar H$ yields a Gaussian random vector
  with covariance matrix
  \[
  \Lambda^{-T}\Omega^{-1}\Lambda^{-1} = 
  \begin{pmatrix}
    \diag(1/\gamma_{i,\{i\}}^2) &
    \diag(1/\gamma_{i,\{i\}}^2)\Gamma_2(\gamma)\\ 
    \Gamma_2(\gamma)^T\diag(1/\gamma_{i,\{i\}}^2)&
    I+\Gamma_2(\gamma)^T
    \diag(1/\gamma_{i,\{i\}}^2)\Gamma_2(\gamma)
  \end{pmatrix},
  \]
  where we have reused the matrices appearing in (\ref{eq:Omega}) and
  (\ref{eq:digraph-covmatrices}).  By standard results about
  conditional distributions of Gaussian random vectors, the covariance
  matrix of the conditional distribution of $\bar Y$ given $\bar H$ is
  the Schur complement
  \begin{multline*}
    \diag(\gamma_{i,\{i\}}^2)^{-1} - \\
    \diag(\gamma_{i,\{i\}}^2)^{-1}\Gamma_2(\gamma)
    \left(I+\Gamma_2(\gamma)^T
      \diag(\gamma_{i,\{i\}}^2)^{-1}\Gamma_2(\gamma) \right)^{-1}
    \Gamma_2(\gamma)^T\diag(\gamma_{i,\{i\}}^2)^{-1}.
  \end{multline*}
  It follows from the matrix inversion lemma that the inverse of the
  conditional covariance matrix is 
  \[
  \diag(\gamma_{i,\{i\}}^2) + \Gamma_2(\gamma)\Gamma_2(\gamma)^T = \phi_\Delta(\gamma).
  \qedhere
  \]
\end{proof}

According to Proposition~\ref{prop:hidden-condition}, positive definite
matrices in $\im(\phi_\Delta)$ also arise as inverses of conditional
covariance matrices.  Zeros in the inverse of the covariance matrix of a
Gaussian random vector have an appealing interpretation in terms of
conditional independence; see again \cite[Chap.~3]{Drton:2009}.

\medskip

\section*{Acknowledgments}
We thank Anton Leykin, Sonja Petrovic, Bernd Sturmfels, and Caroline
Uhler for helpful discussions and anonymous referees for detailed
comments and for suggesting a simpler alternative proof of Lemma
\ref{lem:bernstein}, which we had previously proven by applying
Bernstein's theorem.  Josephine Yu was supported by an NSF
postdoctoral research fellowship.  Mathias Drton was supported by the
NSF under Grant No.~DMS-0746265 and by an Alfred P. Sloan Fellowship.

\bibliographystyle{amsalpha}
\bibliography{covgraphs}

\newcommand{\etalchar}[1]{$^{#1}$}
\def\cprime{$'$}
\providecommand{\bysame}{\leavevmode\hbox to3em{\hrulefill}\thinspace}
\providecommand{\MR}{\relax\ifhmode\unskip\space\fi MR }
\providecommand{\MRhref}[2]{%
  \href{http://www.ams.org/mathscinet-getitem?mr=#1}{#2}
}
\providecommand{\href}[2]{#2}
\begin{thebibliography}{AHMR88}

\bibitem[AHMR88]{AHMR}
Jim Agler, J.~William Helton, Scott McCullough, and Leiba Rodman,
  \emph{Positive semidefinite matrices with a given sparsity pattern},
  Proceedings of the {V}ictoria {C}onference on {C}ombinatorial {M}atrix
  {A}nalysis ({V}ictoria, {BC}, 1987), vol. 107, 1988, pp.~101--149.
  \MR{MR960140 (90h:15030)}

\bibitem[Bar08]{Barber:2008}
David Barber, \emph{Clique matrices for statistical graph decomposition and
  parameterising restricted positive definite matrices}, {P}roceedings of the
  24th {C}onference in {U}ncertainty in {A}rtificial {I}ntelligence (David~A.
  McAllester and Petri Myllym\"{a}ki, eds.), AUAI Press, 2008, pp.~26--33.

\bibitem[CDS95]{Cvetkovic1995}
Drago{\v{s}}~M. Cvetkovi{\'c}, Michael Doob, and Horst Sachs, \emph{Spectra of
  graphs}, third ed., Johann Ambrosius Barth, Heidelberg, 1995, Theory and
  applications. \MR{MR1324340 (96b:05108)}

\bibitem[CW96]{CoxWermuth1996}
D.~R. Cox and Nanny Wermuth, \emph{Multivariate dependencies}, Monographs on
  Statistics and Applied Probability, vol.~67, Chapman \& Hall, London, 1996,
  Models, analysis and interpretation. \MR{MR1456990 (98m:62003)}

\bibitem[DP07]{DrtonSJS2007}
Mathias Drton and Michael~D. Perlman, \emph{Multiple testing and error control
  in {G}aussian graphical model selection}, Statist. Sci. \textbf{22} (2007),
  no.~3, 430--449. \MR{MR2416818}

\bibitem[DSS09]{Drton:2009}
Mathias Drton, Bernd Sturmfels, and Seth Sullivant, \emph{Lectures on algebraic
  statistics}, Birkh\"{a}user Verlag, Basel, Switzerland, 2009.

\bibitem[HPR89]{HPR}
J.~W. Helton, S.~Pierce, and L.~Rodman, \emph{The ranks of extremal positive
  semidefinite matrices with given sparsity pattern}, SIAM J. Matrix Anal.
  Appl. \textbf{10} (1989), no.~3, 407--423. \MR{MR1003106 (90j:05094)}

\bibitem[Lau01]{Laurent2001}
Monique Laurent, \emph{On the sparsity order of a graph and its deficiency in
  chordality}, Combinatorica \textbf{21} (2001), no.~4, 543--570. \MR{MR1863577
  (2002i:05080)}

\bibitem[PDB07]{Palomo2007}
Jesus Palomo, David~B. Dunson, and Ken Bollen, \emph{Bayesian structural
  equation modeling}, Handbook of Latent Variable and Related Models (Sik-Yum
  Lee, ed.), Elsevier, Amsterdam, 2007, pp.~163--188.

\bibitem[PPS89]{Paulsen1989}
Vern~I. Paulsen, Stephen~C. Power, and Roger~R. Smith, \emph{Schur products and
  matrix completions}, J. Funct. Anal. \textbf{85} (1989), no.~1, 151--178.
  \MR{MR1005860 (90j:46051)}

\bibitem[RS02]{Richardson2002}
Thomas Richardson and Peter Spirtes, \emph{Ancestral graph {M}arkov models},
  Ann. Statist. \textbf{30} (2002), no.~4, 962--1030. \MR{MR1926166
  (2003h:60017)}

\bibitem[SRM{\etalchar{+}}98]{SpirtesEtAl1998}
Peter Spirtes, Thomas Richardson, Christopher Meek, Richard Scheines, and Clark
  Glymour, \emph{Using path diagrams as a structural equation modelling tool},
  Sociological Methods and Research \textbf{27} (1998), 182--225.

\bibitem[STD10]{Sullivant:trek}
Seth Sullivant, Kelli Talaska, and Jan Draisma, \emph{Trek separation for
  {G}aussian graphical models}, Ann. Statist. \textbf{38} (2010), no.~3,
  1665--1685.

\end{thebibliography}

\end{document}